\documentclass{SCAE}%SCAEOL for online version; SCAE for publication version; SCAES for the paper dedicated to somebody.
\numberwithin{equation}{section}
\def\lf{\left}
\def\ri{\right}

\def\dbar{\bar\partial}

\def\vv<#1>{\langle#1\rangle}

\def\bi{{\bar i}}
\def\bj{{\bar j}}
\def\bk{{\bar k}}

\def\bl{{\bar l}}
\def\bla{{\bar \lambda}}
\def\bmu{{\bar\mu}}
\def\lam{{\lambda}}

\def\bm{{\bar m}}

\def\XXint#1#2{\setbox0=\hbox{$#1{#2}{\int}$}{#2}\kern-.5\wd0 }

\def\XXint#1#2#3{{\setbox0=\hbox{$#1{#2#3}{\int}$}
     \vcenter{\hbox{$#2#3$}}\kern-.5\wd0}}

\def\vv<#1>{\langle#1\rangle}
\def\ol{\overline}
\def\na{{\nabla}}
\def\cs#1{\left(#1\right)}
\newtheorem{thm}{Theorem}[section]

\newtheorem{lem}{Lemma}[section]
\newtheorem{prop}{Proposition}[section]
\newtheorem{cor}{Corollary}[section]
\theoremstyle{definition}
\newtheorem{defn}{Definition}[section]
\theoremstyle{remark}

\newtheorem{rem}{Remark}[section]

\numberwithin{equation}{section}
\begin{document}

%Basic Information
\Year{2016} %
\Month{January}
\Vol{59} %
\No{1} %
\BeginPage{1} %
\EndPage{XX} %
\AuthorMark{Last1 F N {\it et al.}}
\ReceivedDay{November 17, 2014}
\AcceptedDay{January 22, 2015}
\PublishedOnlineDay{; published online January 22, 2016}
\DOI{10.1007/s11425-000-0000-0} % The author doesn't need fill in it.

% \title[short text for running head]{full title}{comments for title}
\title[Curvature Identities on Almost Hermitian Manifolds and Applications]{Curvature Identities on Almost Hermitian Manifolds and Applications}{}

% \author[]{Full name}{footnote}
% Remark:  One \author for one author

%\author[1,2]{LASTNAME1 FirstName1}{}
\author[1]{Yu Chengjie}{Corresponding author}
%\author[4]{LASTNAME3 FirstName3}{}

%
\address[{\rm1}]{Department of Mathematics, Shantou University, Shantou, Guangdong {\rm 515063}, China;}
%\address[{\rm2}]{Department of Mathematics, University2, City2 {\rm 100002}, Country2;}
%\address[{\rm3}]{Department of Mathematics, University3, City3 {\rm100003}, Country3;}
%\address[{\rm4}]{College of Science, University4, City4 {\rm100004}, Country4}
\Emails{ cjyu@stu.edu.cn}\maketitle

%     Abstract is required.

 {\begin{center}
\parbox{14.5cm}{\begin{abstract}
 In this paper, we systematically derive the Bianchi identities for the canonical connection on an almost Hermitian manifold. Moreover, we also compute the curvature tensor of the Levi-Civita connection on almost Hermitian manifolds in terms of curvature and torsion of the canonical connection. As applications of the curvature identities, we obtain some results about the integrability of quasi K\"ahler manifolds and nearly K\"ahler manifolds.\vspace{-3mm}
\end{abstract}}\end{center}}

%  Keyword is required.
 \keywords{Almost-Hermitian manifold, canonical connection, integrability}

%  \subjclass is required.
 \MSC{53B25, 53C15}

%%%%%%%%%%%%%%%%%%%%%%%%%%%%%%%%%%%%%%%%%%%%%%%%%%%%%%%%%%%%
\renewcommand{\baselinestretch}{1.2}
\begin{center} \renewcommand{\arraystretch}{1.5}
{\begin{tabular}{lp{0.8\textwidth}} \hline \scriptsize
{\bf Citation:}\!\!\!\!&\scriptsize Yu Chengjie. \makeatletter\@titlehead.
Sci China Math, 2016, 59,
 %\@Year, \@Vol: \@BeginPage--\@EndPage,
 doi:~\@DOI\makeatother\vspace{1mm}
\\
\hline
\end{tabular}}\end{center}

%%%%%%%%%%%%%%%%%%%%%%%%%%%%%%%%%%%%%%%%%%%%%%%%%%%%%%%%%%%%
%% Text of article.
%%%%%%%%%%%%%%%%%%%%%%%%%%%%%%%%%%%%%%%%%%%%%%%%%%%%%%%%%%%%
%    Section headings
\baselineskip 11pt\parindent=10.8pt  \wuhao
\section{Introduction}
Almost Hermitian manifolds are almost complex manifolds equipped with a Riemannian metric compatible with the almost complex structure. They form the largest class of generalized K\"ahler manifolds. All the other generalized K\"ahler manifolds such as almost K\"ahler, quasi K\"ahler, nearly K\"ahler and semi K\"ahler manifolds are all special almost Hermitian manifolds. Among all these classes of generalized K\"ahler manifolds, almost K\"ahler and nearly K\"ahler manifolds attracted the  most
attentions, because the former one is related to sympletic geometry and the latter one is nowadays related to theoretical physics.

For geometry of almost Hermitian manifolds, nearly K\"ahler manifolds and almost K\"ahler manifolds with respect to the Levi-Civita connection, one can refer to \cite{AD, Goldberg,Gray1,Gray2,Gray3,Gray4,Sekigawa}. There is another connection more related to the almost complex structure than the Levi-Civita connection which generalizes the Chern connection on Hermitian manifolds (see \cite{Chern}). The connection is called the canonical connection for almost Hermitian manifolds which was first introduced by Ehresmann-Libermann\cite{e}. The geometry of almost Hermitian manifolds with respect to the canonical connection received some attentions since the work of Kobayashi \cite{k,k2}. In \cite{twy}, Tossati, Weinkove and Yau used the canonical connection to solve the Calabi-Yau equation on sympletic manifolds with certain positivity on a combination of curvature tensor and torsion of the canonical connection. Their work is related to Donaldson's program \cite{D2,Donaldson}. The canonical connection is also useful in the study of strictly nearly K\"ahler manifolds by Nagy(\cite{Na1,Na2}) and Butruille \cite{B}. The geometry of almost Hermitian manifolds with respect to the canonical connection was also studied by Tossati \cite{vt} and Tam-Fan-Yu \cite{FTY}. The work of Tossati \cite{vt} extended Yau's Schwartz lemma \cite{Yau} to almost Hermitian manifolds. The work of Tam-Fan-Yu \cite{FTY} extended the result of Seshadri-Zheng \cite{SZ} to almost Hermitian manifolds which is also a generalization of a result in \cite{TY}.

In this paper, we first systematically compute the Bianchi identities on almost Hermitian manifolds. Some of the Bianchi identities listed in this paper are hidden in different forms in \cite{Gray2,twy,k,k2}.  Then, by the help of the Bianchi identities, we compute the curvature of the Levi-Civita connection on an almost Hermitian manifold in terms of curvature and torsion of the canonical connection. In \cite{SV}, the authors made a converse computation for quasi K\"ahler manifolds. Indeed, they compute the curvature of the canonical connection in terms of curvature of the Levi-Civita connection for quasi K\"ahler manifolds. Hence, the curvature identities for quasi K\"ahler manifolds we obtained in section 3 are also hidden in a converse form in \cite{SV}. The curvature identities comparing curvature tensors of the canonical connection and the Levi-Civita connection for Hermitian manifolds are also obtained in \cite{YZ}.
\cite{YZ} also provides some interesting examples of generalized K\"ahler manifolds.

Then, with the help of the curvature identities, we obtain the following two integrable results for quasi K\"ahler manifolds.

\begin{thm}
Let $(M,J,g)$ be a quasi K\"ahler manifold. Then
\begin{equation}
S^c\leq S^*
\end{equation}
all over $M$. Moreover, if the equality holds all over $M$, then $R_{i\bj kl}=0$
for all $i,j,k$ and $l$ all over $M$. If the manifold is almost K\"ahler, then it must be K\"ahler when the equality holds all over $M$. Here $S^*$ is the $*$-scalar curvature for the Levi-Civita connection and $S^c$ is the scalar curvature of the canonical connection.
\end{thm}
\begin{thm}Let $(M,J,g)$ be a compact quasi K\"ahler manifold with quasi positive second Ricci curvature and parallel (2,0)-part of the curvature tensor. Then, the manifold must be K\"ahler.
\end{thm}
For the first result above, there are some related discussions on almost K\"ahler manifolds in \cite{SV}. For the second result above, one should note that without any curvature assumption, even for almost K\"ahler manifolds, the vanishing of (2,0)-part of the curvature tensor of the canonical connection does not imply integrability. One can find such kind of examples in \cite{AD}. Moreover, the assumption that the second Ricci curvature is quasi positive can not be relaxed to nonnegative. Indeed, in \cite{SV}, the authors constructed quasi K\"ahler structures on the Iwasawa manifold with vanishing curvature tensor for the canonical connection.

Finally, by the help of the curvature identities, we obtain the following integrability of nearly K\"ahler manifolds.
\begin{thm}
Let $(M,J,g)$ be a nearly K\"ahler manifold. If the Ricci curvature of the canonical connection is positive definite or negative definite at some point, then the manifold must be K\"ahler.
\end{thm}
\begin{thm}
Let $(M^6,J,g)$ be a strictly nearly K\"ahler manifold. Then $R_{i\bj}=0$ for all $i$ and $j$.
\end{thm}

Here, a strictly nearly K\"ahler manifold is a nearly K\"ahler manifold that is not K\"ahler. By the last theorem, it may be natural to view a six dimensional strictly nearly K\"ahler manifold, such as $\mathbb{S}^6$, as an extension of Calabi-Yau manifolds.
Moreover, as a corollary, we reproduce the well-known result of Gray \cite{Gray2} that any six dimensional strictly nearly K\"ahler manifold is Einstein as a Riemannian manifold.

The organization of the remaining parts of the paper is as follows. In Section 2, we derive Bianchi identities on almost Hermitian manifolds and their corollaries for generalized K\"ahler manifolds. In Section 3, we compare the curvature tensors of the canonical connection and the Levi-Civita connection of an almost Hermitian manifold. In Section 4 and Section 5, we obtain some integrability results for quasi K\"ahler manifolds and nearly K\"ahler manifolds respectively.

\section{Bianchi identities on almost Hermitian manifolds}
In this section, we systematically derive the Bianchi identities on almost Hermitian manifolds.

\begin{defn}[\cite{k,k2,g}] Let $(M,J)$ be an almost complex manifold. A Riemannian metric $g$ on $M$ such that $g(JX,JY)=g(X,Y)$ for any two tangent vectors $X$ and $Y$ is called an almost Hermitian metric.  The triple $(M,J,g)$ is called an almost Hermitian manifold. The two form $\omega_g=g(JX,Y)$ is called the fundamental form of the almost Hermitian manifold. A connection $\nabla$ on an almost Hermitian manifold $(M,J,g)$ such that $\nabla g=0$ and $\nabla J=0$ is called an almost Hermitian connection.
\end{defn}

Let $\nabla$ be a connection on the manifold $M$. Recall that the torsion $\tau$ of the connection is a vector-valued two-form defined as
\begin{equation}
\tau(X,Y)=\nabla_XY-\nabla_YX-[X,Y].
\end{equation}
 There are many almost Hermitian connections on an almost Hermitian manifold. However, there is a unique one such that $\tau(X,\overline Y)=0$ for any two $(1,0)$-vectors $X$ and $Y$. Such a notion is first introduced by Ehresman and Libermann \cite{e}.
\begin{defn}[\cite{k,k2}]The unique almost Hermitian connection $\nabla$ on an almost Hermitian manifold $(M,J,g)$ with vanishing $(1,1)$-part of the torsion is called the canonical connection of the almost Hermitian manifold.
\end{defn}

For sake of convenience, in the remaining parts of this paper, we adopt the following conventions:
\begin{enumerate}
\item without further indications, the manifold $M$  is of real dimension $2n$;
\item $D$ denotes the Levi-Civita connection, $R^L$ denotes its curvature tensor, and ',' means taking covariant derivatives with respect to $D$;
\item $\nabla$ denotes the canonical connection,$R$ denotes the curvature tensor of $\nabla$ and ''$;$'' means taking covariant derivatives with respect to $\na$.
\item without further indications, capital English letters such as $A,B,C$  denote indices in $\{1,\bar1,2,\bar 2,\cdots,n,\bar n\}$;
\item without further indications, $i,j,k$ etc. denote indices in $\{1,2,\cdots,n\}$.
\item without further indications, Greek letters such as $\lambda,\mu$ denote summation indices going through $\{1,2,\cdots,n\}$.
\end{enumerate}

Recall that the Nijenhuis tensor for an almost complex manifold is a vector value two-form defined as
\begin{equation}
N(X,Y)=[JX,JY]-J[JX,Y]-J[X,JY]-[X,Y]
\end{equation}
for any tangent vectors $X$ and $Y$.

The following relation of Nijenhuis tensor and torsion is well know.
\begin{lem}[\cite{e,k,k2}]\label{lem-Nijenhui-torsion} Let $(M,g,J)$ be an almost Hermitian manifold, $(e_1,e_2,\cdots,e_n)$ be a local $(1,0)$-frame. Then $N_{ij}^k=N_{i\bj}^\bk=N_{i\bj}^k=0$ and $N_{ij}^\bk=4\tau_{ij}^\bk$
for all $i,j$ and $k$.
\end{lem}

Recall the definition of curvature operator:
\begin{equation}
R(X,Y)Z=\nabla_X\nabla_YZ-\nabla_Y\nabla_XZ-\nabla_{[X,Y]}Z.
\end{equation}
The curvature tensor is defined as
\begin{equation}
R(X,Y,Z,W)=\vv<R(Z,W)X,Y>.
\end{equation}
Fixed a unitary $(1,0)$-frame $(e_1,e_2,\cdots,e_n)$, since $\nabla J=0$, we have
\begin{equation}\label{eqn-id-curv}
R_{ijAB}=R_{i\ AB}^{\ \bj}=0
\end{equation}
for all indices $i,j$ and $A,B$. Moreover, similarly as in the Riemannian case, we have the following symmetries of the curvature tensor:
\begin{equation}\label{eqn-curv-sym}
R_{ABCD}=-R_{BACD}=-R_{ABDC}
\end{equation}
for all indices $A,B,C$ and $D$. Recall that $R'_{AB}=g^{\bmu\lam}R_{\lam\bmu AB}$ and  $R''_{i\bj}=g^{\bmu\lam}R_{i\bj\lam\bmu}$ are called the first and the second Ricci curvature of the almost Hermitian metric $g$ respectively.

The following general first and second Bianchi identities can be found in \cite{KN}.
\begin{lem}[First Bianchi identity] \label{lem-1-bian}Let $M$ a smooth manifold and $\nabla$ be an affine connection on $M$ with torsion $\tau$. Then
\begin{equation}
\begin{split}
&R(X,Y)Z+R(Y,Z)X+R(Z,X)Y\\
=&(\nabla_X\tau)(Y,Z)+(\nabla_Y\tau)(Z,X)+(\nabla_Z\tau)(X,Y)-\tau(X,\tau(Y,Z))-\tau(Y,\tau(Z,X))-\tau(Z,\tau(X,Y))
\end{split}
\end{equation}
for any tangent vectors $X,Y$ and $Z$.
\end{lem}

\begin{lem}[Second Bianchi identity] \label{lem-2-bian}Let $(M,g)$ be a Riemannian manifold, $\nabla$ be an affine connection compatible with the Riemannian metric $g$, and $\tau$ be the torsion of $\nabla$. Then,
\begin{equation}
\begin{split}
&(\nabla_WR)(X,Y,U,V)+(\nabla_UR)(X,Y,V,W)+(\nabla_VR)(X,Y,W,U)\\
=&-R(X,Y,\tau(U,V),W)-R(X,Y,\tau(V,W),U)-R(X,Y,\tau(W,U),V)
\end{split}
\end{equation}
for any tangent vectors $X,Y,U,V,W$.
\end{lem}

By directly applying the Bianchi identities above, \eqref{eqn-id-curv} and \eqref{eqn-curv-sym}, we have the following identities. Some of them can also be found in different forms in \cite{k,k2} and \cite{twy}.

\begin{cor}\label{cor-first-bian-2}
Let $(M,J,g)$ be an almost Hermitian manifold and fix a local unitary $(1,0)$ frame. Then
\begin{enumerate}
\item $\tau_{ik;l}^\bj+\tau_{kl;i}^\bj+\tau_{li;k}^\bj=\tau_{i\lam}^\bj\tau_{kl}^\lam+\tau_{k\lam}^\bj\tau_{li}^\lam+\tau_{l\lam}^\bj\tau_{ik}^\lam$;
\item $R_{i\bj k\bar l}-R_{k\bj i\bar l}=\tau^j_{ik;\bar
l}-\tau^{\bar \lambda}_{ik}\tau^j_{\bar l\bar\lambda}$;
\item $R_{i\bj  k\bl }-R_{i \bl k\bj }=\tau^\bi_{\bj\bl;
k}-\tau^\bi_{k\lambda}\tau^{\lambda}_{\bj\bl}$;
\item $R_{i\bj k\bl}-R_{k\bl i\bj}=\tau_{ik;\bj}^l+\tau_{\bj\bl;k}^\bi-\tau_{k\lam}^\bi\tau_{\bj\bl}^\lam-\tau_{\bj\bla}^l\tau_{ik}^\bla$;
\item $R_{i\bj kl}=-\tau_{kl;\bj}^\bi+\tau_{\bj\bla}^\bi\tau_{kl}^\bla$;
\item $R_{i\bj kl}+R_{k\bj li}+R_{l\bj ik}=\tau^j_{ik;l}+\tau^j_{kl;i}+\tau^j_{li;k}-\tau^j_{i\lambda}\tau^\lambda_{kl}-\tau^j_{k\lambda}\tau^\lambda_{li}-\tau^j_{l\lambda}\tau^\lambda_{ik}$;
\item
$R_{i\bj kl;m}+R_{i\bj lm;k}+R_{i\bj mk;l}=\tau_{kl}^\lam R_{i\bj m\lam}+\tau_{lm}^\lam R_{i\bj k\lam}+\tau_{mk}^\lam R_{i\bj l\lam}+\tau_{kl}^\bla R_{i\bj m\bla}+\tau_{lm}^\bla R_{i\bj k\bla}+\tau_{mk}^\bla R_{i\bj l\bla}$;
\item $R_{i\bj k\bl;m}-R_{i\bj m\bl;k}=-R_{i\bj mk;\bl}-\tau_{mk}^\lam R_{i\bj\lam\bl}-\tau_{mk}^\bla R_{i\bj\bla\bl}$;
\item $R_{i\bj k\bl;\bm}-R_{i\bj k\bm;\bl}=-R_{i\bj \bl\bm;k}-\tau_{\bl\bm}^\lam R_{i\bj\lam k}+\tau_{\bl\bm}^\bla R_{i\bj k\bla}$;
\end{enumerate}
\end{cor}
\begin{proof}
By Lemma \ref{lem-1-bian}, we have
\begin{equation}
R_{C\ AB}^{\ D}+R_{A\ BC}^{\ D}+R_{B\ CA}^{\ D}=\tau_{BC;A}^D+\tau_{CA;B}^D+\tau_{AB;C}^D-\tau_{AE}^D\tau_{BC}^E-\tau_{BE}^D\tau_{CA}^E-\tau_{CE}^D\tau_{AB}^E.
\end{equation}
Letting $C=i,A=k,B=l$ and $D=\bj$, we obtain (1). Letting $C=i,A=k,B=\bl$ and $D=j$, we obtain(2). Taking conjugate of (2), we obtain (3). Subtracting (2) and (3), we obtain (4). Letting $C=\bi,A=k,B=l$ and $D=\bj$, we obtain (5). Letting $C=i,A=k,B=l$ and $D=j$, we obtain (6).

Moreover, by Lemma \ref{lem-2-bian}, we have
\begin{equation}
R_{ABCD;E}+R_{ABDE;C}+R_{ABEC;D}=-\tau_{CD}^FR_{ABFE}-\tau_{DE}^FR_{ABFC}-\tau_{EC}^FR_{ABFD}.
\end{equation}
Letting $A=i,B=\bj,C=k,D=l$ and $E=m$, we obtain (6). Letting $A=i,B=\bj, C=k,D=\bl$ and $E=m$, we obtain (7). Finally, letting $A=i,B=\bj, C=k,D=\bl$ and $E=\bar{m}$, we obtain (8).
\end{proof}
By Lemma \ref{lem-Nijenhui-torsion}, when the complex structure is integrable, we have $\tau_{ij}^\bk=0$. Hence, we have the following identities on Hermitian manifolds.
\begin{cor}\label{cor-Bian-hermitian}
Let $(M,J,g)$ be a Hermitian manifold and fix a local unitary $(1,0)$-frame. Then
\begin{enumerate}
\item $R_{i\bj k\bar l}-R_{k\bj i\bar l}=\tau^j_{ik;\bar
l}$;
\item $R_{i\bj  k\bl }-R_{i \bl k\bj }=\tau^\bi_{\bj\bl;
k}$;
\item $R_{i\bj k\bl}-R_{k\bl i\bj}=\tau_{ik;\bj}^l+\tau_{\bj\bl;k}^\bi$;
\item $R_{i\bj kl}=0$;
\item $\tau^j_{ik;l}+\tau^j_{kl;i}+\tau^j_{li;k}=\tau^j_{i\lambda}\tau^\lambda_{kl}+\tau^j_{k\lambda}\tau^\lambda_{li}+\tau^j_{l\lambda}\tau^\lambda_{ik}$;

\item $R_{i\bj k\bl;m}-R_{i\bj m\bl;k}=-\tau_{mk}^\lam R_{i\bj\lam\bl}$;
\item $R_{i\bj k\bl;\bm}-R_{i\bj k\bm;\bl}=\tau_{\bl\bm}^\bla R_{i\bj k\bla}$.
\end{enumerate}
\end{cor}

Recall that an almost Hermitian manifold $(M,J,g)$ is called almost K\"ahler if $d\omega_g=0$ and it is called quasi K\"ahler if $\dbar \omega_g=0$. It was shown in \cite{twy}(see also \cite{k,k2}) that quasi Kahlerity is equivalent to $\tau_{ij}^k=0$ for all $i,j$ and $k$, and almost Kahlerity is equivalent to $\tau_{ij}^k=0$ and $\tau_{ij}^\bk+\tau_{ki}^\bj+\tau_{jk}^\bi=0$ for all $i,j$ and $k$ when a local unitary (1,0)-frame is fixed. Hence, we have the following corollary.
\begin{cor}\label{cor-bian-quasi-kahler}
Let $(M,J,g)$ be a quasi K\"ahler manifold and fix a local unitary (1,0)-frame. Then,
\begin{enumerate}
\item $\tau_{ik;l}^\bj+\tau_{kl;i}^\bj+\tau_{li;k}^\bj=0$;
\item $R_{i\bj k\bar l}-R_{k\bj i\bar l}=-\tau^{\bar \lambda}_{ik}\tau^j_{\bar l\bar\lambda}$;
\item $R_{i\bj  k\bl }-R_{i \bl k\bj }=-\tau^\bi_{k\lambda}\tau^{\lambda}_{\bj\bl}$;
\item $R_{i\bj k\bl}-R_{k\bl i\bj}=-\tau_{k\lam}^\bi\tau_{\bj\bl}^\lam-\tau_{ik}^\bla\tau_{\bj\bla}^l$;
\item $R_{i\bj kl}=-\tau_{kl;\bj}^\bi$;
\item $R_{i\bj kl}+R_{k\bj li}+R_{l\bj ik}=0$;
\item
$R_{i\bj kl;m}+R_{i\bj lm;k}+R_{i\bj mk;l}
=\tau_{kl}^\bla R_{i\bj m\bla}+\tau_{lm}^\bla R_{i\bj k\bla}+\tau_{mk}^\bla R_{i\bj l\bla}$;

\item $R_{i\bj k\bl;m}-R_{i\bj m\bl;k}=-R_{i\bj mk;\bl}-\tau_{mk}^\bla R_{i\bj\bla\bl}$;
\item $R_{i\bj k\bl;\bm}-R_{i\bj k\bm;\bl}=-R_{i\bj \bl\bm;k}-\tau_{\bl\bm}^\lam R_{i\bj\lam k}$.
\end{enumerate}
\end{cor}

Recall that an almost Hermitian manifold $(M,J,g)$ is said to be nearly K\"ahler if $(D_XJ)X=0$ for any tangent vector field $X$. The following criterion for nearly K\"ahlerity is well known, see for example \cite{Na1,Na2}.
\begin{lem}\label{lem-criterion-nearly-kahler}
 An almost Hermitian manifold $(M,J,g)$ is nearly K\"ahler if and only if $\tau_{ij}^k=0$ and $\tau_{ij}^\bk=\tau_{jk}^\bi$ for all $i,j$ and $k$ when a local unitary (1,0)-frame is fixed.
\end{lem}

It turns out that the torsion for a nearly K\"ahler manifold must be parallel. This fact was first shown by Kirichenko \cite{Ki} and was crucial for the study of nearly K\"ahler manifolds in \cite{Na1,Na2}. We give a proof of this fact using the curvature identities we have derived in Corollary \ref{cor-bian-quasi-kahler}.

\begin{thm}[Kirichenko]\label{thm-para-tor}
Let $(M,J,g)$ be a nearly K\"ahler manifold. Then $\nabla\tau=0$.
\end{thm}
\begin{proof}
Fix a local unitary (1,0)-frame, by Lemma \ref{lem-criterion-nearly-kahler} we only need to show that $\tau_{ij;l}^\bk=\tau_{ij;\bl}^\bk=0$ for all $i,j,k$ and $l$.

By Lemma \ref{lem-criterion-nearly-kahler} and (5) in  Corollary \ref{cor-bian-quasi-kahler}, we know that
\begin{equation}
R_{i\bj kl}=-\tau_{kl;\bj}^\bi.
\end{equation}
Substituting this into (6) of Corollary \ref{cor-bian-quasi-kahler} and using Lemma \ref{lem-criterion-nearly-kahler}, we have
\begin{equation}
3\tau_{kl;\bj}^\bi=\tau_{kl;\bj}^\bi+\tau_{li;\bj}^\bk+\tau_{ik;\bj}^\bl=0.
\end{equation}
So $\tau_{ij;\bl}^\bk=0$ for all $i,j,k$ and $l$.

On the other hand, by (1) in Corollary \ref{cor-bian-quasi-kahler} and Lemma \ref{lem-criterion-nearly-kahler},
\begin{equation}
\begin{split}
\tau_{ij;k}^\bl=-\tau_{jk;i}^\bl-\tau_{ki;j}^\bl=-\tau_{lj;i}^\bk-\tau_{il;j}^\bk=\tau_{ji;l}^\bk=-\tau_{ij;l}^\bk.
\end{split}
\end{equation}
Therefore,
\begin{equation}
2\tau_{ij;l}^\bk=\tau_{ij;l}^\bk-\tau_{ij;k}^\bl=\tau_{jk;l}^\bi-\tau_{jl;k}^\bi=-\tau_{kj;l}^\bi-\tau_{jl;k}^\bi=\tau_{lk;j}^\bi.
\end{equation}
So, by the last identity,
\begin{equation}
\tau_{kl;j}^\bi=-2\tau_{ij;l}^\bk=4\tau_{kl;j}^\bi.
\end{equation}
and $\tau_{kl;j}^\bi=0$ for all $i,j, k$ and $l$.
\end{proof}

By Lemma \ref{lem-criterion-nearly-kahler} and Theorem \ref{thm-para-tor}, we have the following Bianchi identities for nearly K\"ahler manifolds.
\begin{cor}\label{cor-bian-nearly-kahler}
Let $(M,J,g)$ be a nearly manifold and fix a local unitary $(1,0)$-frame. Then,
\begin{enumerate}
\item $R_{i\bj k\bar l}-R_{k\bj i\bar l}=-\tau^{\bar \lambda}_{ik}\tau^\lam_{\bj\bar l}$;
\item $R_{i\bj  k\bl }-R_{i \bl k\bj }=-\tau^{\bar \lambda}_{ik}\tau^\lam_{\bj\bar l}$;
\item $R_{i\bj k\bl}-R_{k\bl i\bj}=0$;
\item $R'_{i\bj}=R''_{i\bj}$;
\item $R_{i\bj kl}=0$;
\item
$\tau_{kl}^\bla R_{i\bj m\bla}+\tau_{lm}^\bla R_{i\bj k\bla}+\tau_{mk}^\bla R_{i\bj l\bla}=0$;

\item $R_{i\bj k\bl;m}-R_{i\bj m\bl;k}=0$;
\item $R_{i\bj k\bl;\bm}-R_{i\bj k\bm;\bl}=0$.
\end{enumerate}
\end{cor}
\begin{rem}
(1),(2),(7),(8) in different forms can be also found in\cite{Vezzoni2}.
\end{rem}
Since the fist and second Ricci curvature tensors coincide for nearly K\"ahler manifolds, we simply denote them as $R_{i\bj}$ for nearly K\"ahler manifolds.
\section{Curvatures of Levi-Civita and Canonical Connections}
In this section, we compare the curvature tensor $R^L$ of the Levi-Civita connection $D$ and the curvature tensor $R$ of the canonical connection $\na$ on an almost Hermitian manifold $(M,J,g)$.

Recall the following comparison of Levi-Civita connection and canonical connection on almost Hermitian manifolds. One can find in \cite{g} and in \cite{FTY} for a proof.
\begin{lem}\label{lem-comp-connection}
$$\vv<D_YX,Z>=\vv<\nabla_Y X,Z>+\frac{1}{2}\lf(\vv<\tau(X,Y),Z>+\vv<\tau(Y,Z),X>-\vv<\tau(Z,X),Y>\ri)$$
for any tangent vector fields $X,Y$ and $Z$.
\end{lem}

\begin{thm}\label{thm-curv-1}
Let $(M,J,g)$ be an almost Hermitian manifold and fix a local unitary $(1,0)$-frame. Then,
\begin{equation}
\begin{split}
&R^L_{i\bj k\bl}\\
=&\frac{1}{2}(R_{i\bl k\bj}+R_{k\bj i\bl})-\frac{1}{4}(\tau_{\bl\bla}^\bi\tau_{ k\lam}^j+\tau_{i\lam}^l\tau_{\bj\bla}^\bk-\tau_{ik}^\lam\tau_{\bj\bl}^\bla)-\frac{1}{2}(\tau_{k\lam}^\bi\tau_{\bj\bl}^\lam+\tau_{ik}^\bla\tau_{\bl\bla}^j)+\frac{1}{4}(\tau_{ik}^\bla+\tau_{k\lam}^\bi-\tau_{\lam i}^\bk)(\tau_{\bj\bl}^\lam+\tau_{\bl\bla}^j-\tau_{\bla\bj}^l).\\
\end{split}
\end{equation}
\end{thm}
\begin{proof} For $p\in M$, let $(e_1,e_2,\cdots,e_n)$ be local $(1,0)$-frame at $p$ such that $\nabla e_i(p)=0$ and $g_{i\bj}(p)=\delta_{ij}$ (See \cite{Yu2} for the existence of such frames). Then
\begin{equation}\label{eqn-lie}
[e_k,\ol{e_l}](p)=\nabla_{e_k}\ol{e_l}(p)-\nabla_{\ol{e_l}}e_k(p)-\tau(e_k,\ol{e_l})(p)=0.
\end{equation}

By Lemma \ref{lem-comp-connection}, we have
\begin{equation}\label{eqn-com-con-1}
D_{\ol{e_l}}e_i(p)=\frac{1}{2}\tau_{i\lam}^l(p)\ol{e_\lam}+\frac{1}{2}\tau_{\bl\bla}^{\bi}(p)e_\lam
\end{equation}
and
\begin{equation}\label{eqn-com-con-2}
D_{e_k}e_i(p)=\frac{1}{2}\tau_{ik}^\lam e_\lam+\frac{1}{2}\lf(\tau_{ik}^\bla+\tau_{k\lam}^\bi-\tau_{\lam i}^\bk\ri)\ol{e_\lam}.
\end{equation}
Hence
\begin{equation}\label{eqn-curv-1}
\begin{split}
&\vv<D_{e_k}D_{\ol{e_l}}e_i,\ol{e_j}>(p)\\
=&\vv<\nabla_{e_k}D_{\ol{e_l}}e_i,\ol{e_j}>(p)+\frac{1}{2}\lf(\vv<\tau(D_{\ol{e_l}}e_i,e_k),\ol{e_j}>-\vv<\tau(\ol{e_j},D_{\ol{e_l}}e_i),e_k>\ri)(p)\\
=&e_k\vv<D_{\ol{e_l}}e_i,\ol{e_j}>(p)-\frac{1}{4}\tau_{ k\lam}^j\tau_{\bl\bla}^\bi(p)-\frac{1}{4}\tau_{i\lam}^l\tau_{\bj\bla}^\bk(p)\\
=&e_k\left(\vv<\na_{\ol{e_l}}e_i,\ol{e_j}>+\frac{1}{2}\vv<\tau(\ol{e_l},\ol{e_j}),e_i>\ri)(p)-\frac{1}{4}\tau_{ k\lam}^j\tau_{\bl\bla}^\bi(p)-\frac{1}{4}\tau_{i\lam}^l\tau_{\bj\bla}^\bk(p)\\
=&\vv<\na_{e_k}\na_{\ol{e_l}}e_i,\ol{e_j}>(p)-\frac{1}{2}\tau_{\bj\bl;k}^\bi(p)-\frac{1}{4}\tau_{ k\lam}^j\tau_{\bl\bla}^\bi(p)-\frac{1}{4}\tau_{i\lam}^l\tau_{\bj\bla}^\bk(p)\\
\end{split}
\end{equation}
where we have used Lemma \ref{lem-comp-connection}, \eqref{eqn-com-con-1} and \eqref{eqn-com-con-2}.

Similarly, we have
\begin{equation}\label{eqn-curv-2}
\begin{split}
&\vv<D_{\ol{e_l}}D_{e_k}e_i,\ol{e_j}>(p)\\
=&\vv<\na_{\ol{e_l}}D_{e_k}e_i,\ol{e_j}>(p)+\frac{1}{2}(\vv<\tau(D_{e_k}e_i,\ol{e_l}),\ol{e_j}>+\vv<\tau(\ol{e_l},\ol{e_j}),D_{e_k}e_i>-\vv<\tau(\ol{e_j},D_{e_k}e_i),\ol{e_l}>)(p)\\
=&\ol{e_l}\vv<D_{e_k}e_i,\ol{e_j}>-\frac{1}{4}(\tau_{ik}^\bla+\tau_{k\lam}^\bi-\tau_{\lam i}^\bk)(\tau_{\bj\bl}^\lam+\tau_{\bl\bla}^j-\tau_{\bla\bj}^l)-\frac{1}{4}\tau_{\bj\bl}^\bla\tau_{ik}^\lam\\
=&\ol{e_l}\lf(\vv<\nabla_{e_k}e_i,\ol{e_j}>(p)+\frac{1}{2}\vv<\tau(e_i,e_k),\ol{e_j}>\ri)-\frac{1}{4}(\tau_{ik}^\bla+\tau_{k\lam}^\bi-\tau_{\lam i}^\bk)(\tau_{\bj\bl}^\lam+\tau_{\bl\bla}^j-\tau_{\bla\bj}^l)-\frac{1}{4}\tau_{\bj\bl}^\bla\tau_{ik}^\lam\\
=&\vv<\na_{\ol{e_l}}\na_{e_k}e_i,\ol{e_j}>+\frac{1}{2}\tau_{ik;\bl}^j-\frac{1}{4}(\tau_{ik}^\bla+\tau_{k\lam}^\bi-\tau_{\lam i}^\bk)(\tau_{\bj\bl}^\lam+\tau_{\bl\bla}^j-\tau_{\bla\bj}^l)-\frac{1}{4}\tau_{\bj\bl}^\bla\tau_{ik}^\lam.\\
\end{split}
\end{equation}
So,
\begin{equation}\label{eqn-curv-5}
\begin{split}
&R^L_{i\bj k\bl}(p)\\
=&\vv<D_{e_k}D_{\ol{e_l}}e_i-D_{\ol{e_l}}D_{e_k}e_i,\ol{e_j}>(p)\\
=&R_{i\bj k\bl}-\frac{1}{2}(\tau_{\bj\bl;k}^\bi+\tau_{ik;\bl}^j)-\frac{1}{4}(\tau_{ k\lam}^j\tau_{\bl\bla}^\bi+\tau_{i\lam}^l\tau_{\bj\bla}^\bk)+\frac{1}{4}(\tau_{ik}^\bla+\tau_{k\lam}^\bi-\tau_{\lam i}^\bk)(\tau_{\bj\bl}^\lam+\tau_{\bl\bla}^j-\tau_{\bla\bj}^l)+\frac{1}{4}\tau_{\bj\bl}^\bla\tau_{ik}^\lam\\
=&\frac{1}{2}(R_{i\bl k\bj}+R_{k\bj i\bl})-\frac{1}{2}(\tau_{k\lam}^\bi\tau_{\bj\bl}^\lam+\tau_{\bl\bla}^j\tau_{ik}^\bla)-\frac{1}{4}(\tau_{ k\lam}^j\tau_{\bl\bla}^\bi+\tau_{i\lam}^l\tau_{\bj\bla}^\bk-\tau_{\bj\bl}^\bla\tau_{ik}^\lam)+\frac{1}{4}(\tau_{ik}^\bla+\tau_{k\lam}^\bi-\tau_{\lam i}^\bk)(\tau_{\bj\bl}^\lam+\tau_{\bl\bla}^j-\tau_{\bla\bj}^l)\\
\end{split}
\end{equation}
where we have used (2),(3) in Corollary \ref{cor-first-bian-2},\eqref{eqn-lie}, \eqref{eqn-curv-1} and \eqref{eqn-curv-2}.
\end{proof}
By directly using the last identity, we have the following identities for holomorphic sectional curvature.
\begin{cor} Let $(M,J,g)$ be an almost Hermitian manifold and fix a local unitary $(1,0)$-frame. Then
\begin{equation}
R^L_{i\bi i\bi}=R_{i\bi i\bi}+\tau_{i\lam}^\bi\tau_{\bi\bla}^i-\frac{1}{2}\tau_{i\lam}^i\tau_{\bi\bla}^\bi.
\end{equation}
Moreover, when the manifold is Hermitian,
\begin{equation}
R^L_{i\bi i\bi}=R_{i\bi i\bi}-\frac{1}{2}\tau_{i\lam}^i\tau_{\bi\bla}^\bi.
\end{equation}
So, the holomorphic sectional curvature of the Levi-Civita connection is not greater than the holomorphic sectional curvature of the Chern connection.

When the manifold is quasi K\"ahler,
\begin{equation}
R^L_{i\bi i\bi}=R_{i\bi i\bi}+\tau_{i\lam}^\bi\tau_{\bi\bla}^i.
\end{equation}
So, the holomorphic sectional curvature of the Levi-Civita connection is not less than the holomorphic sectional curvature of the canonical connection. Furthermore, when the manifold is nearly K\"ahler,
\begin{equation}
R^L_{i\bi i\bi}=R_{i\bi i\bi}
\end{equation}
which means that the holomorphic sectional curvature of the Levi-Civita connection is the same as the holomorphic sectional curvature of the canonical connection.
\end{cor}
By that when the complex structure is integrable, $\tau_{ij}^\bk=0$, we have the following corollary.
\begin{cor}
Let $(M,J,g)$ be a Hermitian manifold and fixed a local unitary  $(1,0)$-frame. Then
\begin{equation}
R^L_{i\bj k\bl}=\frac{1}{2}(R_{i\bl k\bj}+R_{k\bj i\bl})-\frac{1}{4}\cs{\tau_{\bl\bla}^\bi\tau_{ k\lam}^j+\tau_{i\lam}^l\tau_{\bj\bla}^\bk-\tau_{ik}^\lam\tau_{\bj\bl}^\bla}.
\end{equation}
\end{cor}
By the second equality in \eqref{eqn-curv-5}, we have the following corollary.

\begin{cor}\label{cor-quasi-kahler-1}
Let $(M,J,g)$ be a quasi K\"ahler manifold and fixe a local unitary $(1,0)$-frame. Then
\begin{equation}
R^L_{i\bj k\bl}=R_{i\bj k\bl}+\frac{1}{4}\cs{\tau_{ik}^\bla+\tau_{k\lam}^\bi-\tau_{\lam i}^\bk}\cs{\tau_{\bj\bl}^\lam+\tau_{\bl\bla}^j-\tau_{\bla\bj}^l}.
\end{equation}
So,
$$R^L(X,\ol X,Y,\ol Y)\geq R(X,\ol X,Y,\ol Y)$$
for any $(1,0)$-vectors $X$ and $Y$.
\end{cor}
\begin{rem}
A similar identity was also obtained in \cite{SV}.
\end{rem}

By noting that for an almost K\"ahler manifold, $\tau_{ij}^\bk+\tau_{jk}^\bi+\tau_{ki}^\bj=0$. We have the following corollary.
\begin{cor}\label{cor-almost-kahler-1}
Let $(M,J,g)$ be an almost K\"ahler manifold and fix a unitary frame. Then
\begin{equation}
R^L_{i\bj k\bl}=R_{i\bj k\bl}+\tau_{\lam i}^\bk\tau_{\bla\bj}^l.
\end{equation}
\end{cor}
By Lemma \ref{lem-criterion-nearly-kahler}, we have the following curvature identity for nearly K\"ahler manifolds.
\begin{cor}\label{cor-nearly-kahler-1}
Let $(M,J,g)$ be a nearly K\"ahler manifold and fix a local unitary $(1,0)$-frame. Then
\begin{equation}
R^L_{i\bj k\bl}=R_{i\bj k\bl}+\frac{1}{4}\tau_{ik}^\bla\tau_{\bj\bl}^\lam.
\end{equation}

\end{cor}

\begin{thm}\label{thm-curv-2}
Let $(M,J,g)$ be an almost Hermitian manifold and fix a local unitary $(1,0)$-frame. Then
\begin{equation}
\begin{split}
R^L_{ijk\bl}=&\frac{1}{2}\cs{R_{k\bl ij}-R_{i\bl jk}-R_{j\bl ki}}+\frac{1}{2}\cs{\tau_{ij;k}^l-\tau_{ij}^\bla\tau_{\bl\bla}^\bk}+\frac{1}{4}\cs{\tau_{j\lam}^l\tau_{ik}^\lam-\tau_{i\lam}^l\tau_{jk}^\lam}
\\
&+\frac{1}{4}\tau_{\bl\bla}^\bi(\tau_{jk}^\bla-\tau_{k\lam }^\bj+\tau_{\lam j}^\bk)-\frac{1}{4}\tau_{\bl\bla}^\bj\cs{\tau_{ik}^\bla-\tau_{k\lam}^\bi+\tau_{\lam i}^\bk}.
\end{split}
\end{equation}
\end{thm}
\begin{proof} For $p\in M$ and $(e_1,e_2,\cdots,e_n)$ the local $(1,0)$-frame at $p$ as in the proof of Theorem \ref{thm-curv-1}. Then
\begin{equation}\label{eqn-curv-3}
\begin{split}
&\vv<D_{e_k}D_{\ol{e_l}}e_i,e_j>(p)\\
=&\vv<\na_{e_k}D_{\ol{e_l}}e_i,e_j>+\frac{1}{2}(\vv<\tau(D_{\ol{e_l}}e_i,e_k),e_j>+\vv<\tau(e_k,e_j),D_{\ol{e_l}}e_i>-\vv<\tau(e_j,D_{\ol{e_l}}e_i),e_k>)\\
=&e_k\vv<D_{\ol{e_l}}e_i,e_j>-\frac{1}{4}\tau_{\bl\bla}^\bi(\tau_{j k}^\bla+\tau_{k\lam}^\bj-\tau_{\lam j}^\bk)-\frac{1}{4}\tau_{i\lam}^l\tau_{jk}^\lam\\
=&e_k\left(\vv<\na_{\ol{e_l}}e_i,e_j>-\frac{1}{2}\vv<\tau(e_j,e_i),\ol{e_l}>\ri)-\frac{1}{4}\tau_{\bl\bla}^\bi(\tau_{j k}^\bla+\tau_{k\lam}^\bj-\tau_{\lam j}^\bk)-\frac{1}{4}\tau_{i\lam}^l\tau_{jk}^\lam\\
=&\frac{1}{2}\tau_{ji;k}^l(p)-\frac{1}{4}\tau_{\bl\bla}^\bi(\tau_{j k}^\bla+\tau_{k\lam}^\bj-\tau_{\lam j}^\bk)-\frac{1}{4}\tau_{i\lam}^l\tau_{jk}^\lam\\
\end{split}
\end{equation}
where we have used \eqref{eqn-com-con-1}. Moreover
\begin{equation}\label{eqn-curv-4}
\begin{split}
&\vv<D_{\ol{e_l}}D_{e_k}e_i,e_j>(p)\\
=&\vv<\na_{\ol{e_l}}D_{e_k}e_i,e_j>(p)+\frac{1}{2}(\vv<\tau(D_{e_k}e_i,\ol{e_l}),e_j>-\vv<\tau(e_j,D_{e_k}e_i),\ol{e_l}>)(p)\\
=&\ol{e_l}\cs{\vv<\na_{e_k}e_i,e_j>+\frac{1}{2}\cs{\vv<\tau(e_i,e_k),e_j>+\vv<\tau(e_k,e_j),e_i>-\vv<\tau(e_j,e_i),e_k>}}(p)\\
&-\frac{1}{4}\tau_{\bl\bla}^\bj\cs{\tau_{ik}^\bla+\tau_{k\lam}^\bi-\tau_{\lam i}^\bk}(p)-\frac{1}{4}\tau_{j\lam}^l\tau_{ik}^\lam(p)\\
=&\frac{1}{2}\cs{\tau_{ik;\bl}^\bj+\tau_{kj;\bl}^\bi-\tau_{ji;\bl}^\bk}(p)-\frac{1}{4}\tau_{\bl\bla}^\bj\cs{\tau_{ik}^\bla+\tau_{k\lam}^\bi-\tau_{\lam i}^\bk}(p)-\frac{1}{4}\tau_{j\lam}^l\tau_{ik}^\lam(p)\\
=&-\frac{1}{2}\cs{R_{j\bl ik}+R_{i\bl kj}-R_{k\bl ji}}(p)+\frac{1}{2}\cs{\tau_{\bl\bla}^\bj\tau_{ik}^\bla+\tau_{\bl\bla}^\bi\tau_{kj}^\bla-\tau_{\bl\bla}^\bk\tau_{ji}^\bla}(p)\\
&-\frac{1}{4}\tau_{\bl\bla}^\bj\cs{\tau_{ik}^\bla+\tau_{k\lam}^\bi-\tau_{\lam i}^\bk}(p)-\frac{1}{4}\tau_{j\lam}^l\tau_{ik}^\lam(p)\\
\end{split}
\end{equation}
where we have used Corollary \ref{cor-first-bian-2} and \eqref{eqn-com-con-2}.

Combining \eqref{eqn-curv-3} and \eqref{eqn-curv-4}, we get
\begin{equation}
\begin{split}
&R^L_{ijk\bl}(p)\\
=&\frac{1}{2}\cs{R_{j\bl ik}+R_{i\bl kj}-R_{k\bl ji}}+\frac{1}{2}\cs{\tau_{ij;k}^l-\tau_{ij}^\bla\tau_{\bl\bla}^\bk}+\frac{1}{4}\cs{\tau_{j\lam}^l\tau_{ik}^\lam-\tau_{i\lam}^l\tau_{jk}^\lam}
\\
&-\frac{1}{4}\tau_{\bl\bla}^\bi(-\tau_{jk}^\bla+\tau_{k\lam }^\bj-\tau_{\lam j}^\bk)+\frac{1}{4}\tau_{\bl\bla}^\bj\cs{-\tau_{ik}^\bla+\tau_{k\lam}^\bi-\tau_{\lam i}^\bk}.
\end{split}
\end{equation}
\end{proof}
Noting that $\tau_{ij}^\bk=0$ for Hermitian manifolds, we have the following corollary.
\begin{cor}\label{cor-curv-hermitian-2}
Let $(M,J,g)$ be a Hermitian manifold and fix a local unitary $(1,0)$-frame. Then
\begin{equation}
R^L_{ijk\bl}=\frac{1}{2}\tau_{ij;k}^l+\frac{1}{4}\cs{\tau_{j\lam}^l\tau_{ik}^\lam-\tau_{i\lam}^l\tau_{jk}^\lam}.
\end{equation}
\end{cor}

 Applying the properties of quasi K\"ahler manifolds that   $\tau_{ij}^k=0$ and (6) in Corollary \ref{cor-bian-quasi-kahler}, we obtain the following corollary.
\begin{cor}\label{cor-quasi-kahler-2}
Let $(M,J,g)$ be a quasi K\"ahler manifold and fix a local unitary $(1,0)$-frame. Then
\begin{equation}
R^L_{ij k\bl}=R_{k\bl ij}.
\end{equation}
\end{cor}
\begin{rem}
The same identity was also obtained in \cite{SV}.
\end{rem}
By (5) in Corollary \ref{cor-bian-quasi-kahler}, we have the following corollary for nearly K\"ahler manifolds.
\begin{cor}\label{cor-nearly-kahler-2}
Let $(M,J,g)$ be a nearly K\"ahler manifold and fix a local unitary $(1,0)$-frame. Then
\begin{equation}
R^L_{ijk\bl}=0.
\end{equation}
\end{cor}
\begin{thm}
Let $(M,J,g)$ be an almost Hermitian manifold and fix a local unitary $(1,0)$-frame. Then,
\begin{equation}
\begin{split}
R^L_{ijkl}=&\frac{1}{2}\cs{\tau_{kl;i}^\bj-\tau_{kl;j}^\bi}+\frac{1}{2}\cs{\tau_{ij;k}^\bl-\tau_{ij;l}^\bk}+\frac{1}{2}\cs{\tau_{ij}^\lam\tau_{kl}^\bla+\tau_{ij}^\bla\tau_{kl}^\lam}+\frac{1}{4}\tau_{ik}^\lam\cs{\tau_{jl}^\bla-\tau_{ l\lam }^\bj-\tau_{\lam j}^\bl}+\frac{1}{4}\tau_{jl}^\lam\cs{\tau_{ik}^\bla-\tau_{k\lam}^\bi-\tau_{\lam i}^\bk}\\
&-\frac{1}{4}\tau_{il}^\lam\cs{\tau_{jk}^\bla-\tau_{ k\lam }^\bj-\tau_{\lam j}^\bk}-\frac{1}{4}\tau_{jk}^\lam\cs{\tau_{il}^\bla-\tau_{l\lam}^\bi-\tau_{\lam i}^\bl}.\\
\end{split}
\end{equation}
\begin{proof}
We proceed similarly as before. Let $p\in M$ and $(e_1,e_2,\cdots,e_n)$ the local $(1,0)$-frame at $p$ in the proof of Theorem \ref{thm-curv-1}. Then
\begin{equation}
\begin{split}
&\vv<D_{e_k}D_{e_l}e_i,e_j>(p)\\
=&\vv<\na_{e_k}D_{e_l}e_i,e_j>+\frac{1}{2}\cs{\vv<\tau(D_{e_l}e_i,e_k),e_j>+\vv<\tau(e_k,e_j),D_{e_l}e_i>-\vv<\tau(e_j,D_{{e_l}}e_i),e_k>}\\
=&e_k\vv<D_{e_l}e_i,e_j>+\frac{1}{4}\tau_{il}^\lam\cs{\tau_{\lam k}^\bj+\tau_{kj}^\bla-\tau_{j\lam}^\bk}+\frac{1}{4}\tau_{kj}^\lam\cs{\tau_{il}^\bla+\tau_{l\lam}^\bi-\tau_{\lam i}^\bl}\\
=&e_k\cs{\vv<\na_{e_l}e_i,e_j>+\frac{1}{2}\cs{\vv<\tau(e_i,e_l),e_j>+\vv<\tau(e_l,e_j),e_i>-\vv<\tau(e_j,e_i),e_l>}}\\
&+\frac{1}{4}\tau_{il}^\lam\cs{\tau_{\lam k}^\bj+\tau_{kj}^\bla-\tau_{j\lam}^\bk}+\frac{1}{4}\tau_{kj}^\lam\cs{\tau_{il}^\bla+\tau_{l\lam}^\bi-\tau_{\lam i}^\bl}\\
=&\frac{1}{2}\cs{\tau_{il;k}^\bj+\tau_{lj;k}^\bi-\tau_{ji;k}^\bl}-\frac{1}{4}\tau_{il}^\lam\cs{\tau_{jk}^\bla+\tau_{ k\lam }^\bj-\tau_{\lam j}^\bk}-\frac{1}{4}\tau_{jk}^\lam\cs{\tau_{il}^\bla+\tau_{l\lam}^\bi-\tau_{\lam i}^\bl}\\
\end{split}
\end{equation}
where we have used \eqref{eqn-com-con-2}. Similarly,
\begin{equation}
\begin{split}
&\vv<D_{e_l}D_{e_k}e_i,e_j>(p)=\frac{1}{2}\cs{\tau_{ik;l}^\bj+\tau_{kj;l}^\bi-\tau_{ji;l}^\bk}-\frac{1}{4}\tau_{ik}^\lam\cs{\tau_{jl}^\bla+\tau_{ l\lam }^\bj-\tau_{\lam j}^\bl}-\frac{1}{4}\tau_{jl}^\lam\cs{\tau_{ik}^\bla+\tau_{k\lam}^\bi-\tau_{\lam i}^\bk}.\\
\end{split}
\end{equation}
Moreover, note that $$[e_k,e_l](p)=\na_{e_k}e_l(p)-\na_{e_l}e_k(p)-\tau(e_k,e_l)(p)=-\tau_{kl}^\lam e_\lam-\tau_{kl}^\bla e_\bla.$$
So
\begin{equation}
\begin{split}
\vv<D_{[e_k,e_l]}e_i,e_j>(p)=&-\tau_{kl}^\lam\vv<D_{e_\lam}e_i,e_j>-\tau_{kl}^\bla\vv<D_{\ol{e_\lam}}e_i,e_j>\\
=&-\frac{1}{2}\tau_{kl}^\lam\cs{\tau_{i\lam}^\bj+\tau_{\lam j}^\bi-\tau_{ji}^\bla}-\frac{1}{2}\tau_{ij}^\lam\tau_{kl}^\bla\\
=&\frac{1}{2}\cs{\tau_{j\lam}^\bi-\tau_{i\lam}^\bj}\tau_{kl}^\lam-\frac{1}{2}\cs{\tau_{ij}^\lam\tau_{kl}^\bla+\tau_{ij}^\bla\tau_{kl}^\lam}
\end{split}
\end{equation}
where we have used \eqref{eqn-com-con-1} and \eqref{eqn-com-con-2}.

Hence
\begin{equation}
\begin{split}
&R^L_{ijkl}(p)\\
=&\vv<D_{e_k}D_{e_l}e_i-D_{e_l}D_{e_k}e_i-D_{[e_k,e_l]}e_i,e_j>(p)\\
=&\frac{1}{2}\cs{\tau_{il;k}^\bj-\tau_{ik;l}^\bj}+\frac{1}{2}\cs{\tau_{lj;k}^\bi-\tau_{kj;l}^\bi}+\frac{1}{2}\cs{\tau_{ij;k}^\bl-\tau_{ij;l}^\bk}\\
&+\frac{1}{4}\tau_{ik}^\lam\cs{\tau_{jl}^\bla+\tau_{ l\lam }^\bj-\tau_{\lam j}^\bl}+\frac{1}{4}\tau_{jl}^\lam\cs{\tau_{ik}^\bla+\tau_{k\lam}^\bi-\tau_{\lam i}^\bk}\\
&-\frac{1}{4}\tau_{il}^\lam\cs{\tau_{jk}^\bla+\tau_{ k\lam }^\bj-\tau_{\lam j}^\bk}-\frac{1}{4}\tau_{jk}^\lam\cs{\tau_{il}^\bla+\tau_{l\lam}^\bi-\tau_{\lam i}^\bl}-\frac{1}{2}\cs{\tau_{j\lam}^\bi-\tau_{i\lam}^\bj}\tau_{kl}^\lam+\frac{1}{2}\cs{\tau_{ij}^\lam\tau_{kl}^\bla+\tau_{ij}^\bla\tau_{kl}^\lam}\\
=&\frac{1}{2}\cs{\tau_{kl;i}^\bj-\cs{\tau_{i\lam}^\bj\tau_{kl}^\lam+\tau_{k\lam}^\bj\tau_{li}^\lam+\tau_{l\lam}^\bj\tau_{ik}^\lam}}+\frac{1}{2}\cs{\tau_{lk;j}^\bi-\cs{\tau_{l\lam}^\bi\tau_{kj}^\lam+\tau_{k\lam}^\bi\tau_{jl}^\lam+\tau_{j\lam}^\bi\tau_{lk}^\lam}}\\
&+\frac{1}{2}\cs{\tau_{ij;k}^\bl-\tau_{ij;l}^\bk}+\frac{1}{4}\tau_{ik}^\lam\cs{\tau_{jl}^\bla+\tau_{ l\lam }^\bj-\tau_{\lam j}^\bl}+\frac{1}{4}\tau_{jl}^\lam\cs{\tau_{ik}^\bla+\tau_{k\lam}^\bi-\tau_{\lam i}^\bk}\\
&-\frac{1}{4}\tau_{il}^\lam\cs{\tau_{jk}^\bla+\tau_{ k\lam }^\bj-\tau_{\lam j}^\bk}-\frac{1}{4}\tau_{jk}^\lam\cs{\tau_{il}^\bla+\tau_{l\lam}^\bi-\tau_{\lam i}^\bl}-\frac{1}{2}\cs{\tau_{j\lam}^\bi-\tau_{i\lam}^\bj}\tau_{kl}^\lam+\frac{1}{2}\cs{\tau_{ij}^\lam\tau_{kl}^\bla+\tau_{ij}^\bla\tau_{kl}^\lam}\\
=&\frac{1}{2}\cs{\tau_{kl;i}^\bj-\tau_{kl;j}^\bi}+\frac{1}{2}\cs{\tau_{ij;k}^\bl-\tau_{ij;l}^\bk}+\frac{1}{2}\cs{\tau_{ij}^\lam\tau_{kl}^\bla+\tau_{ij}^\bla\tau_{kl}^\lam}\\
&+\frac{1}{4}\tau_{ik}^\lam\cs{\tau_{jl}^\bla-\tau_{ l\lam }^\bj-\tau_{\lam j}^\bl}+\frac{1}{4}\tau_{jl}^\lam\cs{\tau_{ik}^\bla-\tau_{k\lam}^\bi-\tau_{\lam i}^\bk}-\frac{1}{4}\tau_{il}^\lam\cs{\tau_{jk}^\bla-\tau_{ k\lam }^\bj-\tau_{\lam j}^\bk}-\frac{1}{4}\tau_{jk}^\lam\cs{\tau_{il}^\bla-\tau_{l\lam}^\bi-\tau_{\lam i}^\bl}\\
\end{split}
\end{equation}
where we have used (1) in Corollary \ref{cor-Bian-hermitian}.
\end{proof}
\end{thm}
As before, using properties of Hermitian manifolds and quasi K\"ahler manifolds, we have the following corollaries.
\begin{cor}\label{cor-hermitian-3}
Let $(M,g,J)$ be a Hermitian manifold and fix a local unitary $(1,0)$-frame. Then, $R^L_{ijkl}=0$.
\end{cor}
\begin{cor}\label{cor-quasi-kahler-3}
Let $(M,J,g)$ be a quasi K\"ahler manifold and fix a local unitary $(1,0)$-frame. Then
\begin{equation}
R^L_{ijkl}=\frac{1}{2}\cs{\tau_{kl;i}^\bj-\tau_{kl;j}^\bi}+\frac{1}{2}\cs{\tau_{ij;k}^\bl-\tau_{ij;l}^\bk}.
\end{equation}
\end{cor}
By Theorem \ref{thm-para-tor}, we have the following corollary for nearly K\"ahler manifolds.
\begin{cor}\label{cor-nearly-kahler-3}
Let $(M,J,g)$ be a nearly K\"ahler manifold and fix a local unitary $(1,0)$-frame. Then $R^L_{ijkl}=0$.
\end{cor}

At the end of this section, we compute the Ricci curvature of the Levi-Civita connection of an almost K\"ahler manifold and a nearly K\"ahler manifold in terms of curvature and torsion for the canonical connection.
\begin{thm}\label{thm-almost-kahler-ricci}
Let $(M,J,g)$ be an almost K\"ahler manifold and fixed a local unitary $(1,0)$-frame. Then,
\begin{equation}
R^{L}_{ij}=R_{i\bla\lam j}+R_{j\bla \lam i}\ \mbox{and}\ R^L_{i\bj}=R'_{i\bj}-2\tau_{i\mu}^\bla\tau_{\bj\bla}^\mu.
\end{equation}
%and
%\begin{equation}
%\begin{split}
%R^L_{i\bj}=&R'_{i\bj}-2\tau_{i\mu}^\bla\tau_{\bj\bla}^\mu\\
%\end{split}
%\end{equation}
\end{thm}
\begin{proof}
\begin{equation}
\begin{split}
R^L_{ij}=R^L_{\lam ij\bla}+R^L_{\bla ij\lam}=R^L_{\lam ij\bla}+R^L_{\lam ji\bla}=R_{j\bla \lam i}+R_{i\bla \lam j}\\
\end{split}
\end{equation}
where we have used the symmetries for the curvature tensor of Levi-Civita connection and Corollary \ref{cor-quasi-kahler-2}.
Moreover
\begin{equation}
\begin{split}
R^L_{i\bj}=&R^L_{\lam i\bj\bla}+R^L_{\bla i\bj\lam}\\
=&-R^L_{\lam \bj\bla i}-R^L_{\lam\bla i\bj}+R^L_{\bla i\bj\lam}\\
=&2R^L_{\lam \bj i\bla}-R^L_{\lam\bla i\bj}\\
=&2R_{\lam\bj i\bla}+2\tau_{\mu\lam}^\bi\tau_{\bmu\bj}^\lam-R'_{i\bj}-\tau_{\mu\lam}^\bi\tau_{\bmu\bla}^j\\
=&2(R_{\lam\bla i\bj}-\tau_{i\mu}^\bla\tau_{\bj\bla}^\mu)-R'_{i\bj}+\tau_{\mu\lam}^\bi\tau_{\bmu\bj}^\lam+\tau_{\mu\lam}^\bi(\tau_{\bmu\bj}^\lam+\tau_{\bla\bmu}^j)\\
=&R'_{i\bj}-2\tau_{i\mu}^\bla\tau_{\bj\bla}^\mu+\tau_{\mu\lam}^\bi\tau_{\bmu\bj}^\lam-\tau_{\mu\lam}^\bi\tau_{\bj\bla}^\mu\\
=&R'_{i\bj}-2\tau_{i\mu}^\bla\tau_{\bj\bla}^\mu+\tau_{\mu\lam}^\bi\tau_{\bmu\bj}^\lam-\tau_{\lam\mu}^\bi\tau_{\bj\bmu}^\lam\\
=&R'_{i\bj}-2\tau_{i\mu}^\bla\tau_{\bj\bla}^\mu
\end{split}
\end{equation}
where we have used Corollary \ref{cor-almost-kahler-1} and Corollary \ref{cor-bian-quasi-kahler}.
\end{proof}
\begin{thm}\label{thm-nearly-kahler-ricci}
Let $(M,J,g)$ be a nearly K\"ahler manifold and fix a local unitary $(1,0)$-frame. Then
\begin{equation}
R^L_{ij}=0\ \mbox{and}\ R^L_{i\bj}=R_{i\bj}+\frac{5}{4}\tau_{i\lam}^\bmu\tau_{\bj\bla}^\mu.
\end{equation}
So,
$$Ric^L(X,X)\geq Ric(X,X)$$
for any real tangent vectors $X$ , where $Ric^L$ means the Ricci curvature tensor of the Levi-Civita connection and $Ric$ means the Ricci curvature tensor of the canonical connection.

\end{thm}
\begin{proof}
The proof is the same as the proof of Theorem \ref{thm-almost-kahler-ricci} using Corollary \ref{cor-nearly-kahler-1} and Corollary \ref{cor-nearly-kahler-2}.
\end{proof}
\section{Integrability of quasi K\"ahler manifolds }
In this section, with the help of the curvature identities derived in the last two sections, we obtain some results about the integrability of quasi K\"ahler manifolds.

First, recall that the $*$-scalar curvature $S^*$ for the Levi-Civita connection of an almost Hermitian manifold is defined as (see for example \cite{AD})
\begin{equation}
S^*=R^L_{\lam\bla\mu\bmu}.
\end{equation}
Let $S^c$ be the scalar curvature of the canonical connection. That is,
\begin{equation}
S^c=R_{\lam\bla\mu\bmu}.
\end{equation}
\begin{thm}
Let $(M,J,g)$ be a quasi K\"ahler manifold. Then
\begin{equation}
S^c\leq S^*
\end{equation}
all over $M$. Moreover, if the equality holds all over $M$, then $R_{i\bj kl}=0$
for all $i,j,k$ and $l$ all over $M$. If the manifold is almost K\"ahler, then it must be K\"ahler when the equality holds all over $M$.
\end{thm}
\begin{proof}
By Corollary \ref{cor-quasi-kahler-1}, we know that
\begin{equation}
S^*=S^c+\frac{1}{4}\sum_{\lam,\mu,\nu=1}^n\lf|\tau_{\lam\mu}^{\bar\nu}+\tau_{\mu\nu}^\bla-\tau_{\nu\lam}^\bmu\ri|^2\geq S^c.
\end{equation}
When the equality holds all over $M$, we have
\begin{equation}\label{eqn-scalar-1}
\tau_{ij}^\bk+\tau_{jk}^\bi-\tau_{ki}^\bj=0
\end{equation}
for all $i,j$ and $k$, all over $M$. Then
\begin{equation}
\tau_{ij;\bl}^\bk+\tau_{jk;\bl}^\bi-\tau_{ki;\bl}^\bi=0
\end{equation}
which means that
\begin{equation}
R_{k\bl ij}+R_{i\bl jk}-R_{j\bl ki}=0
\end{equation}
by (5) in Corollary \ref{cor-bian-quasi-kahler}. Combining this with (6) in Corollary \ref{cor-bian-quasi-kahler}, we know that $R_{i\bj kl}=0$ for all $i,j,k$ and $l$.

When the manifold is almost K\"ahler and the equality holds all over $M$. It is clear that $\tau_{ij}^\bk=0$ for all $i,j$ and $k$ by combining \eqref{eqn-scalar-1} and $\tau_{ij}^\bk+\tau_{jk}^\bi+\tau_{ki}^\bj=0$ for all $i,j$ and $k$.
\end{proof}
\begin{rem}
In \cite{AD}, there is a similar inequality in integration form for compact almost K\"ahler manifolds.
\end{rem}

We know that when the complex structure is integrable, then the (2,0)-part of the curvature tensor for the canonical connection vanishes. One may ask if the converse is true. It turns out that the converse is not true even when the manifold is almost K\"ahler. Indeed, there are examples of strictly almost K\"ahler manifolds (almost K\"ahler but not K\"ahler) with vanishing (2,0)-part of the curvature tensor for the canonical connection which is equivalent to that the curvature tensor for the Levi-Civita connection satisfies the third Gray identity by Corollary \ref{cor-quasi-kahler-2}(see for example \cite{AD}). However, when some curvature conditions are imposed, the answer turns out to be affirmative.

\begin{thm}Let $(M,J,g)$ be a compact quasi K\"ahler manifold with quasi positive second Ricci curvature and parallel (2,0)-part of the curvature tensor for the canonical connection. Then, the manifold must be K\"ahler.
\end{thm}
\begin{proof}

Note that for a quasi K\"ahler manifold, the Laplacian opertors on functions for the canonical connection and the Levi-Civita connection coincide (See for example \cite{twy,FTY}). So, we simply denote it as $\Delta$. Fix a local unitary $(1,0)$-frame, we have
\begin{equation}\label{eqn-integrability}
\begin{split}
\Delta(\tau_{ij}^\bk\tau_{\bi\bj}^k)=&(\tau_{ij}^\bk\tau_{\bi\bj}^k)_{l\bl}\\
=&\tau_{ij;l}^\bk\tau_{\bi\bj;\bl}^k+\tau_{ij;\bl}^\bk\tau_{\bi\bj;l}^k+\tau_{ij;l\bl}^\bk\tau_{\bi\bj}^k+\tau_{ij}^\bk\tau_{\bi\bj;l\bl}^k\\
=&\tau_{ij;l}^\bk\tau_{\bi\bj;\bl}^k+\tau_{ij;\bl}^\bk\tau_{\bi\bj;l}^k+(\tau_{ij;\bl l}^\bk+R_{i\bla l\bl}\tau_{\lam j}^\bk+R_{j\bla l\bl}\tau_{i\lam}^\bk+R_{\lam\bk l\bl}\tau_{ij}^\bla)\tau_{\bi\bj}^k+\tau_{ij}^\bk\tau_{\bi\bj;l\bl}^k\\
=&\tau_{ij;l}^\bk\tau_{\bi\bj;\bl}^k+\tau_{ij;\bl}^\bk\tau_{\bi\bj;l}^k+(-R_{k\bl ij;\l}+R''_{i\bla }\tau_{\lam j}^\bk+R''_{j\bla }\tau_{i\lam}^\bk+R''_{\lam\bk }\tau_{ij}^\bla)\tau_{\bi\bj}^k-\tau_{ij}^\bk R_{\bk l \bi\bj;\bl}\\
=&\tau_{ij;l}^\bk\tau_{\bi\bj;\bl}^k+\tau_{ij;\bl}^\bk\tau_{\bi\bj;l}^k+(R''_{i\bla }\tau_{\lam j}^\bk+R''_{j\bla }\tau_{i\lam}^\bk+R''_{\lam\bk }\tau_{ij}^\bla)\tau_{\bi\bj}^k\\
\geq& 0
\end{split}
\end{equation}
where we have used Corollary \ref{cor-bian-quasi-kahler} and the Ricci identity for commuting covariant derivatives with respect to the canonical connection (See for example \cite{FTY}). By maximum principle, we know that $\|\tau\|^2$ is constant and that $\tau$ is parallel. Since the second Ricci curvature is positive at some point,
$\tau$ vanishes at some point. Therefore $\tau$ vanishes all over $M$ and the metric is K\"ahler.
\end{proof}
%\begin{rem}
%In \cite{SV}, the authors construct a quasi K\"ahler structure on the Iwasawa manifold with vanishing curvature tensor of the canonical connection. We can not obtain integrability by just assuming the second Ricci curvature is nonnegative and vanishing (2,0)-part of curvature tensor for a quasi K\"ahler manifold.
%\end{rem}

Furthermore, we have the following integral inequality by taking integration on \eqref{eqn-integrability}.
\begin{thm} Let $(M,J,g)$ be a compact quasi K\"ahler manifold. Then
\begin{equation}
\int_M\sum_{i,j,k,l=1}(R''_{i\bl }\tau_{l j}^\bk+R''_{j\bl }\tau_{il}^\bk+R''_{l\bk }\tau_{ij}^\bl)\tau_{\bi\bj}^kdV\leq \int_M\sum_{i,j,k,l=1}^n|R_{i\bj kl}|^2dV.
\end{equation}
\end{thm}
\begin{proof}
By Lemma \ref{lem-comp-connection}, it is not hard to check that the divergence operators on vector fields for the canonical connection and the Levi-Civita connection coincide on quasi K\"ahler manifolds. Moreover, by \eqref{eqn-integrability}, we have
\begin{equation}
\begin{split}
\Delta(\tau_{ij}^\bk\tau_{\bi\bj}^k)=&\tau_{ij;l}^\bk\tau_{\bi\bj;\bl}^k+\tau_{ij;\bl}^\bk\tau_{\bi\bj;l}^k+(R''_{i\bla }\tau_{\lam j}^\bk+R''_{j\bla }\tau_{i\lam}^\bk+R''_{\lam\bk }\tau_{ij}^\bla)\tau_{\bi\bj}^k+\tau_{ij;\bl l}^\bk\tau_{\bi\bj}^k+\tau_{ij}^\bk\tau_{\bi\bj;l\bl}^k\\
=&\tau_{ij;l}^\bk\tau_{\bi\bj;\bl}^k-\tau_{ij;\bl}^\bk\tau_{\bi\bj;l}^k+(R''_{i\bla }\tau_{\lam j}^\bk+R''_{j\bla }\tau_{i\lam}^\bk+R''_{\lam\bk }\tau_{ij}^\bla)\tau_{\bi\bj}^k+(\tau_{ij;\bl}^\bk\tau_{\bi\bj}^k)_l+(\tau_{ij}^\bk\tau_{\bi\bj;l}^k)_\bl\\
\geq&-\sum_{i,j,k,l=1}^n|R_{i\bj kl}|^2+(R''_{i\bla }\tau_{\lam j}^\bk+R''_{j\bla }\tau_{i\lam}^\bk+R''_{\lam\bk }\tau_{ij}^\bla)\tau_{\bi\bj}^k+(\tau_{ij;\bl}^\bk\tau_{\bi\bj}^k)_l+(\tau_{ij}^\bk\tau_{\bi\bj;l}^k)_\bl\\
\end{split}
\end{equation}
where we have used Corollary \ref{cor-bian-quasi-kahler}. Taking integration on both sides of the last inequality and applying the divergence theorem, we obtain the conclusion.
\end{proof}
\section{Integrability of nearly K\"ahler manifolds}
In this section, we consider intergrability of nearly K\"ahler manifolds.
\begin{prop}\label{prop-curv-id-nearly-kahler}
Let $(M,J,g)$ be  a nearly K\"ahler manifold and fix a local unitary $(1,0)$-frame. Then
\begin{equation}\label{eqn-id-NK}
\sum_{k,l,\lam,\mu=1}^nR_{i\bj k\bl}\tau_{\bla\bmu}^k\tau_{\lam\mu}^\bl=0
\end{equation}
for all $i$ and $j$.
\end{prop}
\begin{proof}
By Corollary \ref{cor-bian-nearly-kahler}, we know that
\begin{equation}
\sum_{\lam=1}^n\cs{\tau_{kl}^\bla R_{i\bj m\bla}+\tau_{lm}^\bla R_{i\bj k\bla}+\tau_{mk}^\bla R_{i\bj l\bla}}=0
\end{equation}
for all $i,j,k,l$ and $m$. Then
\begin{equation}
\sum_{k,l,m,\lam=1}^n\cs{\tau_{kl}^\bla R_{i\bj m\bla}+\tau_{lm}^\bla R_{i\bj k\bla}+\tau_{mk}^\bla R_{i\bj l\bla}}\tau_{\bk\bl}^m=0.
\end{equation}
By Lemma \ref{lem-criterion-nearly-kahler}, we have
\begin{equation}
\begin{split}
3\sum_{k,l,\lam,\mu=1}^nR_{i\bj k\bl}\tau_{\bla\bmu}^k\tau_{\lam\mu}^\bl=\sum_{k,l,m,\lam=1}\cs{R_{i\bj m\bla}\tau_{kl}^\bla\tau_{\bk\bl}^m+R_{i\bj k\bla}\tau_{lm}^\bla\tau_{\bl\bm}^k+R_{i\bj l\bla}\tau_{mk}^\bla\tau_{\bm\bk}^l}=0.
\end{split}
\end{equation}
Hence
\begin{equation}
\sum_{k,l,\lam,\mu=1}^nR_{i\bj k\bl}\tau_{\bla\bmu}^k\tau_{\lam\mu}^\bl=0.
\end{equation}
\end{proof}
\begin{thm}
Let $(M,J,g)$ be a nearly K\"ahler manifold. Then, if the Ricci curvature of the canonical connection is positive definite or negative definite at some point, then the manifold must be K\"ahler.
\end{thm}
\begin{proof}
By Proposition \ref{prop-curv-id-nearly-kahler}, we have
\begin{equation}
\sum_{k,l,\lam,\mu=1}R_{k\bl}\tau_{\bla\bmu}^k\tau_{\lam\mu}^\bl=0.
\end{equation}
If the $(R_{k\bl})$ is positive or negative at some point $p\in M$, then
\begin{equation}
\tau_{\lam\mu}^\bl(p)=0
\end{equation}
for all $\lam,\mu$ and $l$. Note that $\tau$ is parallel on $M$ by Theorem \ref{thm-para-tor}. So $\tau=0$ all over $M$ and hence $(M,J,g)$ is K\"ahler.
\end{proof}

\begin{thm}
Let $(M^6,J,g)$ be a strictly nearly K\"ahler manifold. Then $R_{i\bj}=0$ for all $i$ and $j$.
\end{thm}
\begin{proof}
Let $\Phi(X,Y,Z)=\vv<\tau(X,Y),Z>$ for any (1,0)-vectors $X,Y$ and $Z$. Then $\Phi$ is a $(3,0)$-form on $M$ by Lemma \ref{lem-criterion-nearly-kahler}. Let $e_1,e_2,e_3$ be a unitary frame and $\omega^1,\omega^2,\omega^3$ be its dual frame. Suppose that $\Phi=c\omega^1\wedge\omega^2\wedge\omega^3$. Since the manifold is non-K\"ahler, $c\neq 0$. Moreover, it is clear that
\begin{equation}
\tau_{ij}^\bk=c\cdot\mbox{sgn}\left(\begin{array}{ccc}1&2&3\\ i&j&k\end{array}\right).
\end{equation}
Substituting it in to \ref{eqn-id-NK}, we have
\begin{equation}
|c|^2R_{i\bj}=0.
\end{equation}
This completes the proof.
\end{proof}
By the last result, we can reproduce the following  well-known result of Gray \cite{Gray2}.
\begin{cor}\label{cor-NK-E}
Let $(M^6,J,g)$ be a strictly nearly K\"ahler manifold. Then, $(M^6,g)$ as a Riemannian manifold must be a Einstein manifold with positive scalar curvature.
\end{cor}
\begin{proof}
Let $\Phi$ and $c$ be the same as the proof of last theorem. Since $\tau$ is parallel by Theorem \ref{thm-para-tor}, $\|\tau\|^2$ is constant. Hence $|c|^2$ does not depend on the point we chosen. By Theorem \ref{thm-nearly-kahler-ricci} and the last theorem, we know that
\begin{equation}
R^L_{ij}=0\ \mbox{and}\ R^L_{i\bj}= \frac{5}{4}|c|^2g_{i\bj}.
\end{equation}
This completes the proof.
\end{proof}

\Acknowledgements{The author would like to thank the referees for helpful and inspiring suggestions. This work was partially supported by GDNSF S2012010010038, the National Natural Science Foundation of
China (11571215) and a supporting project from the Department of Education of Guangdong Province with contract no. Yq2013073.}


\end{document}